\documentclass[a4paper,12pt]{amsart}
\usepackage{latexsym,amsmath,amssymb}

\newtheorem{thm}{Theorem}[section]

\newtheorem{lem}[thm]{Lemma}
\newtheorem{prop}[thm]{Proposition}
\theoremstyle{definition}
\newtheorem{defn}[thm]{Definition}
\theoremstyle{remark}

\numberwithin{equation}{section}

\begin{document}
\title[]{\textsc{\textbf{ Reducibility of WCE operators on $L^2(\mathcal{F})$}}}
\author{\sc\bf Y. Estaremi}
\address{\sc Y. Estaremi}
\email{yestaremi@pnu.ac.ir}

\address{Department of Mathematics, Payame Noor University , P. O. Box: 19395-3697, Tehran,
Iran}

\thanks{}
\thanks{}
\subjclass[2010]{47A15.}
\keywords{conditional expectation operator, invariant subspace, reducing subspace, sigma subalgebra.}
\date{}
\dedicatory{}

\begin{abstract}
In this paper we characterize the closed subspaces of $L^2(\mathcal{F})$ that reduce the operators of the form $E^{\mathcal{A}}M_u$, in which $\mathcal{A}$  is a $\sigma$- subalgebra of $\mathcal{F}$.  We show that $L^2(A)$ reduces $E^{\mathcal{A}}M_u$ if and only if $E^{\mathcal{A}}(\chi_A)=\chi_A$ on the support of $E^{\mathcal{A}}(|u|^2)$, where $A\in \mathcal{F}$. Also, some necessary and sufficient conditions are provided for $L^2(\mathcal{B})$ to reduces $E^{\mathcal{A}}M_u$, for the $\sigma$- subalgebra $\mathcal{B}$ of $\mathcal{F}$.
\end{abstract}

\maketitle

\commby{}


\section{\textsc{Introduction}}

Let $(X,\mathcal{F},\mu)$ be a complete $\sigma$-finite measure space.
For any sub-$\sigma$-finite algebra $\mathcal{A}\subseteq
 \mathcal{F}$, the $L^2$-space $L^2(X,\mathcal{A},\mu_{\mid_{\mathcal{A}}})$ is abbreviated  by
$L^2(\mathcal{A})$, and its norm is denoted by $\|.\|_2$. All
comparisons between two functions or two sets are to be
interpreted as holding up to a $\mu$-null set. The support of a
measurable function $f$ is defined as $S(f)=\{x\in X; f(x)\neq
0\}$. We denote the vector space of all equivalence classes of
almost everywhere finite valued measurable functions on $X$ by
$L^0(\mathcal{F})$. For a sub-$\sigma$-finite algebra
$\mathcal{A}\subseteq\mathcal{F}$, the conditional expectation operator
associated with $\mathcal{A}$ is the mapping $f\rightarrow
E^{\mathcal{A}}f$, defined for all non-negative measurable
function $f$ as well as for all $f\in L^2(\mathcal{F})$, where
$E^{\mathcal{A}}f$, by the Radon-Nikodym theorem, is the unique
$\mathcal{A}$-measurable function satisfying
$$\int_{A}fd\mu=\int_{A}E^{\mathcal{A}}fd\mu, \ \ \ \forall A\in \mathcal{A} .$$
As an operator on $L^{2}({\mathcal{F}})$, $E^{\mathcal{A}}$ is
idempotent and $E^{\mathcal{A}}(L^2(\mathcal{F}))=L^2(\mathcal{A})$.
This operator will play a major role in our work.\\
Composition of conditional expectation operators
and multiplication operators  appear often in the study of other
operators such as multiplication operators and weighted
composition operators. Specifically, in \cite{mo}, S.-T. C. Moy
characterized all operators on $L^p$ of the form $f\rightarrow
E^{\mathcal{A}}(fg)$ for $g$ in $L^q$ with $E^{\mathcal{A}}(|g|)$ bounded. And R. G. Douglas, \cite{dou}, analyzed positive projections on
$L^{1}$ and many of his characterizations are in terms of
combinations of multiplication operators and conditional expectations. Also, P.G. Dodds, C.B. Huijsmans and B. De Pagter, \cite{dhd},
extended these characterizations to the setting of function ideals
and vector lattices. Some other people studied weighted conditional expectation operators on measurable function spaces in \cite{g,her, lam} and references therein. Moreover,  we investigated some classic properties of weighted  conditional expectation
operators of the form $M_wE^{\mathcal{A}}M_u$ on $L^p$ spaces in \cite{e,es,ej}. In this paper we characterize reducible closed subspaces of $L^2(\mathcal{F})$ for the weighted conditional expectation operators of the form $E^{\mathcal{A}}M_u$.

\section{ \textsc{Reducibility of WCE operators}}

Throughout this section we assume that
$u\in\mathcal{D}(E^{\mathcal{A}}):=\{f\in
L^0(\mathcal{F}): E^{\mathcal{A}}(|f|)\in L^0(\mathcal{A})\}$.
Now we recall the definition of weighted conditional expectation operators on $L^2(\mathcal{F})$.
\begin{defn}
Let $(X,\mathcal{F},\mu)$ be a $\sigma$-finite measure space and let $\mathcal{A}$ be a
$\sigma$-subalgebra of $\mathcal{F}$ such that $(X,\mathcal{A},\mu_{\mathcal{A}})$ is
also $\sigma$-finite. Let $E^{\mathcal{A}}$ be the conditional
expectation operator relative to $\mathcal{A}$. If $u\in L^0(\mathcal{F})$ such that $uf$ is conditionable
and $E^{\mathcal{A}}(uf)\in L^{2}(\mathcal{F})$ for all $f\in \mathcal{D}\subseteq L^{2}(\mathcal{F})$,
where $\mathcal{D}$ is a linear subspace, then the corresponding weighted conditional
expectation (or briefly WCE) operator is the linear transformation
$E^{\mathcal{A}}M_u:\mathcal{D}\rightarrow L^{2}(\mathcal{F})$ defined by $f\rightarrow E^{\mathcal{A}}(uf)$.
\end{defn}
For a bounded WCE operator $T=E^{\mathcal{A}}M_u$ on the Hilbert space $L^{2}(\mathcal{F})$ we have
$$T^{\ast}=M_{\bar{u}}E^{\mathcal{A}}, \ \ \ T^{\ast}T=M_{\bar{u}}E^{\mathcal{A}}M_u, \ \ \ TT^{\ast}=E^{\mathcal{A}}M_{E^{\mathcal{A}}(|u|^2)}.$$
In this section we will be concerned with the reducibility of the  WCE operators of the form $E^{\mathcal{A}}M_u$. Recall that an operator $T$ on a Hilbert space $\mathcal{H}$ is reducible if there is a proper closed subspace $M$ of $\mathcal{H}$, other that $\{0\}$ and $\mathcal{H}$, for which $TM\subseteq M$ and $T^{\ast}M\subseteq M$. In this case we say that $M$ reduces $T$. Equivalently, $M$ reduces $T$ if and only if the orthogonal projection $P$ onto $M$ commutes with $T$. We shall make use of the fact that if a projection commutes with $T$, then it commutes with any polynomial in $T$ and $T^{\ast}$. Now we have the next proposition.

\begin{prop}\label{p1} Let $T=E^{\mathcal{A}}M_u$ be a bounded operator on $L^2(\mathcal{F})$, $M\subseteq L^2(\mathcal{F})$ be a subspace such that reduces $E^{\mathcal{A}}M_u$ and $P$ be the orthogonal projection onto $M$. Then we have the followings:

(a) $E^{\mathcal{A}}(uPf)=PE^{\mathcal{A}}(uf)$ for every $f\in L^2(\mathcal{F})$.\\

(b) $\bar{u}E^{\mathcal{A}}(Pf)=P(\bar{u}E^{\mathcal{A}}(f))$ for every $f\in L^2(\mathcal{F})$.\\

(c) $E^{\mathcal{A}}(|u|^2)E^{\mathcal{A}}(Pf)=PE^{\mathcal{A}}(E^{\mathcal{A}}(|u|^2)f)$ for every $f\in L^2(\mathcal{F})$.\\

(d) $\bar{u}E^{\mathcal{A}}(uPf)=P(\bar{u}E^{\mathcal{A}}(uf))$ for every $f\in L^2(\mathcal{F})$.\\

(e) $P(E^{\mathcal{A}}(|u|^2)E^{\mathcal{A}}(uf))=E^{\mathcal{A}}(|u|^2)P(E^{\mathcal{A}}(uf))$ for every $f\in L^2(\mathcal{F})$.

\end{prop}
\begin{proof} The items (a), (b), (c) and (d) follow from the fact that $P$ commute with $T$, $T^{\ast}$, $TT^{\ast}$ and $T^{\ast}T$, respectively. By applying $T$ to the equation in item (c) we get the item (e).
\end{proof}
Now by using Proposition \ref{p1} we give a necessary and sufficient conditions under which the subspace $L^2(A)$ reduces $E^{\mathcal{A}}M_u$ for $A\in \mathcal{F}$.
\begin{thm} Let $E^{\mathcal{A}}M_u$ be bounded on $L^2(\mathcal{F})$ and $A\in \mathcal{F}$. Then $L^2(A)$ reduces $E^{\mathcal{A}}M_u$ if and only if $E^{\mathcal{A}}(\chi_A)=\chi_A$ on $S(E^{\mathcal{A}}(|u|^2))$.
\end{thm}
\begin{proof} If $L^2(A)$ reduces $E^{\mathcal{A}}M_u$, then by the Proposition \ref{p1} part (c) we have $E^{\mathcal{A}}(|u|^2)E^{\mathcal{A}}(\chi_Af)=E^{\mathcal{A}}(|u|^2)\chi_AE^{\mathcal{A}}(f)$ for all $f\in L^2(\mathcal{F})$. Let $\{F_n\}$ be a sequence in $\mathcal{F}$ increasing to $X$ with $\mu(F_n)<\infty$, and put $f_n=\chi_{F_n}$. Then $f_n\in L^2(\mathcal{F})$ and $f_n\uparrow 1$, hence $E^{\mathcal{A}}(f_n)\uparrow 1$ and $E^{\mathcal{A}}(\chi_Af_n)\uparrow E^{\mathcal{A}}(\chi_A)$. Thus we get that $E^{\mathcal{A}}(|u|^2)E^{\mathcal{A}}(\chi_A)=E^{\mathcal{A}}(|u|^2)\chi_A$ and so $E^{\mathcal{A}}(\chi_A)=\chi_A$ on $S(E^{\mathcal{A}}(|u|^2))$. Now we prove the converse, let $E^{\mathcal{A}}(\chi_A)=\chi_A$ on $S(E^{\mathcal{A}}(|u|^2))$. By conditional type Holder inequality we have that $S(E^{\mathcal{A}}(uf))\subseteq S(E^{\mathcal{A}}(|u|^2))$ and $S(E^{\mathcal{A}}(\chi_A)uf)\subseteq S(E^{\mathcal{A}}(|u|^2))$ for all $f\in L^2(\mathcal{F})$. Therefore we have $$\chi_AE^{\mathcal{A}}(uf)=E^{\mathcal{A}}(E^{\mathcal{A}}(\chi_A)uf)=E^{\mathcal{A}}(\chi_Auf),$$
this completes the proof.
\end{proof}
Let $\mathcal{B}\subseteq \mathcal{F}$ be a $\sigma$- subalgebra. In the next theorem we give a sufficient condition under which the subspace $L^2(\mathcal{B})$ reduces $E^{\mathcal{A}}M_u$.

\begin{thm}\label{t1} Let $E^{\mathcal{A}}M_u$ be bounded on $L^2(\mathcal{F})$. Then for every $\sigma$- subalgebra $\mathcal{B}\subseteq \mathcal{F}$ such that $\mathcal{B}\subseteq \mathcal{A}$ or $\mathcal{A}\subseteq \mathcal{B}$ and $u$ is $\mathcal{B}$-measurable, the subspace $L^2(\mathcal{B})$ reduces $E^{\mathcal{A}}M_u$.
\end{thm}
\begin{proof} If we decompose  $L^2(\mathcal{F})$ as a direct sum $L^2(\mathcal{B})\oplus \mathcal{K}(E^{\mathcal{B}})$, then the corresponding block matrix of $E^{\mathcal{A}}M_u$ is as follow:
\begin{center}
$E^{\mathcal{A}}M_u=\left(
   \begin{array}{cc}
     E^{\mathcal{B}}E^{\mathcal{A}}M_uE^{\mathcal{B}} & E^{\mathcal{B}}E^{\mathcal{A}}M_u(I-E^{\mathcal{B}}) \\
     (I-E^{\mathcal{B}})E^{\mathcal{A}}M_uE^{\mathcal{B}} & (I-E^{\mathcal{B}})E^{\mathcal{A}}M_u(I-E^{\mathcal{B}}) \\
   \end{array}
 \right)
$.
 \end{center}
 If $\mathcal{B}\subseteq \mathcal{A}$ or $\mathcal{A}\subseteq \mathcal{B}$, then in both cases we have $E^{\mathcal{B}}E^{\mathcal{A}}=E^{\mathcal{A}}E^{\mathcal{B}}$. Hence we get that
 $$E^{\mathcal{B}}E^{\mathcal{A}}M_u(I-E^{\mathcal{B}})=E^{\mathcal{A}}(E^{\mathcal{B}}M_u-E^{\mathcal{B}}M_uE^{\mathcal{B}}).$$
 Since $u$ is $\mathcal{B}$-measurable, then we have $E^{\mathcal{B}}E^{\mathcal{A}}M_u(I-E^{\mathcal{B}})=0$. Similarly we have $(I-E^{\mathcal{B}})E^{\mathcal{A}}M_uE^{\mathcal{B}}=0$. These imply that $L^2(\mathcal{B})$ reduces $E^{\mathcal{A}}M_u$.
\end{proof}
Let $\mathcal{A},\mathcal{B}\subseteq \mathcal{F}$ be $\sigma$-sub algebras. In the next theorem we determine the relation between $\mathcal{A}$ and $\mathcal{B}$ if $L^2(\mathcal{B})$ reduces $E^{\mathcal{A}}M_u$ and the converse.
\begin{thm}\label{t2} Let $E^{\mathcal{A}}M_u$ be bounded on $L^2(\mathcal{F})$. Then we have the followings:\\

(a) If $\mathcal{B}\subseteq \mathcal{F}$  is a sigma subalgebra such that $E^{\mathcal{B}}E^{\mathcal{A}}=E^{\mathcal{B}\cap\mathcal{A}}$ and $u$ is $\mathcal{B}$-measurable, then the subspace $L^2(\mathcal{B})$ reduces $E^{\mathcal{A}}M_u$.\\

(b) If $L^2(\mathcal{B})$ reduces $E^{\mathcal{A}}M_u$, then $E^{\mathcal{B}}E^{\mathcal{A}}f=E^{\mathcal{A}}E^{\mathcal{B}}f$ for every $f\in L^2(S(E^{\mathcal{A}}(|u|^2)))$ and $u$ is $\mathcal{B}$-measurable.\\

(c) If $S(E^{\mathcal{A}}(|u|^2))=X$, then $L^2(\mathcal{B})$ reduces $E^{\mathcal{A}}M_u$ if and only if $E^{\mathcal{B}}E^{\mathcal{A}}=E^{\mathcal{B}\cap\mathcal{A}}$ and $u$ is $\mathcal{B}$-measurable.
\end{thm}
\begin{proof} (a) The condition $E^{\mathcal{B}}E^{\mathcal{A}}=E^{\mathcal{B}\cap\mathcal{A}}$ implies that $E^{\mathcal{B}}E^{\mathcal{A}}=E^{\mathcal{A}}E^{\mathcal{B}}$. Hence by the same method of Theorem \ref{t1} we get the result.\\

(b) Suppose that  $L^2(\mathcal{B})$ reduces $E^{\mathcal{A}}M_u$. Then by Proposition \ref{p1} part (b) we have $E^{\mathcal{B}}M_{\bar{u}}E^{\mathcal{A}}=M_{\bar{u}}E^{\mathcal{A}}E^{\mathcal{B}}$. Let $\{F_n\}$ be a sequence in $\mathcal{F}$ increasing to $X$ with $\mu(F_n)<\infty$, and put $f_n=\chi_{F_n}$. Then $f_n\in L^2(\mathcal{F})$ and $f_n\uparrow 1$, hence $E^{\mathcal{A}}(f_n)\uparrow 1$. These observations shows that $\bar{u}$ should be $\mathcal{B}$-measurable. Similarly by Proposition \ref{p1} part (c) we get that $E^{\mathcal{A}}(|u|^2)$ is $\mathcal{B}$- measurable. Again by part (c) of Proposition \ref{p1} we get that $E^{\mathcal{B}}E^{\mathcal{A}}f=E^{\mathcal{A}}E^{\mathcal{B}}f$ for every $f\in L^2(S(E^{\mathcal{A}}(|u|^2)))$.\\

(c) It is a direct consequence of parts (a) and (b).

\end{proof}
Here we find a $\sigma$-subalgebra $\mathcal{B}$ of $\mathcal{A}$ such that $L^2(\mathcal{B})$ reduces $E^{\mathcal{A}}M_u$.
\begin{thm}\label{t3}
Let $\mathcal{A}\subseteq \mathcal{F}$ be a $\sigma$-algebra and $E^{\mathcal{A}}M_u$ be a bounded operator on $L^2(\mathcal{F})$. Then there exists a $\sigma$-subalgebra $\mathcal{B}\subseteq \mathcal{F}$ such that $L^2(\mathcal{B})$ reduces $E^{\mathcal{A}}M_u$.
\end{thm}
\begin{proof}
Let $\mathcal{B}$ be the $\sigma$-algebra generated by the sets in $u^{-1}(\mathcal{C})\cup \mathcal{A}$, in which $\mathcal{C}$ is the $\sigma$-algebra of $\mathbb{C}$ such that $u:(X, \mathcal{F})\rightarrow (\mathbb{C},\mathcal{C}) $ is measurable. If $u$ is $\mathcal{A}$-measurable, then $\mathcal{B}\subseteq\mathcal{A}$ and if $u$ is not $\mathcal{A}$-measurable, then $\mathcal{A}\subseteq \mathcal{B}$. In both cases we have $E^{\mathcal{A}}E^{\mathcal{B}}=E^{\mathcal{B}}E^{\mathcal{A}}=E^{\mathcal{A}}$ and $u$ is $\mathcal{B}$-measurable. Hence by Theorem \ref{t1} we get that $L^2(\mathcal{B})$ reduces $E^{\mathcal{A}}M_u$.
\end{proof}

Now we recall a fundamental lemma from general operator theory.

\begin{lem}\label{l1} Let $(X, \mathcal{F}, \mu)$ be a finite measure space, $T\in \mathcal{B}(L^2(\mathcal{F}))$ and $S$ be a closed operator on $L^2(\mathcal{F})$. If $T=S$ on a dense subset of $L^2(\mathcal{F})$, then $S$ is bounded and $T=S$.
\end{lem}
 A $\ast$-subalgebra of $\mathcal{B}(L^2(\mathcal{F}))$ is called a Von Neumann algebra on the Hilbert space $L^2(\mathcal{F})$ if it is closed in strong operator topology (SOT). Now we get a Von Neumann algebra of WCE operators in the next theorem.
\begin{thm}\label{t5}
 If $(X, \mathcal{F}, \mu)$ is a finite measure space and $\mathcal{A}\subset \mathcal{F}$ is a $\sigma$-subalgebra, then $\mathcal{W}=\{E^{\mathcal{A}}M_g: g\in L^{\infty}(\mathcal{A})\}$ is a unital commutative Von Neumann algebra with unit $E^{\mathcal{A}}$.
  \end{thm}
\begin{proof} It is easy to see that $\mathcal{W}$ is a self-adjoint operator subalgebra of $\mathcal{B}(H)$. Then we only need to prove $\mathcal{W}$ is strongly closed. Let $\{E^{\mathcal{A}}M_{u_{\alpha}}\}_{\alpha}\subseteq \mathcal{W}$ and $T\in \mathcal{B}(L^2(\mathcal{F}))$ such that $\|E^{\mathcal{A}}(u_{\alpha}f)-T(f)\|_{L^2}\rightarrow 0$ for all $f\in L^2(\mathcal{F})$. Hence for the constant function $\mathbf{1}$  we have $\|u_{\alpha}-T(\mathbf{1})\|_{L^2}\rightarrow 0$ and so $T(\mathbf{1})$ is $\mathcal{A}$-measurable. Also, for every $f\in L^{\infty}(\mathcal{F})$ and $\alpha$ we have
\begin{align*}
\|T(\mathbf{1})E^{\mathcal{A}}(f)-T(f)\|_{L^2}&\leq \|T(\mathbf{1})E^{\mathcal{A}}(f)-E^{\mathcal{A}}(u_{\alpha}f)\|_{L^2}+\|E^{\mathcal{A}}(u_{\alpha}f)-T(f)\|_{L^2}\\
&\leq \|T(\mathbf{1})-u_{\alpha}\|_{L^2}\|f\|_{\infty}+\|E^{\mathcal{A}}(u_{\alpha}f)-T(f)\|_{L^2}.
\end{align*}
This implies that $T=E^{\mathcal{A}}M_{T(\mathbf{1})}$ on $L^{\infty}(\mathcal{F})$. Since $L^{\infty}(\mathcal{F})$  is dense in $L^2(\mathcal{F})$ and $E^{\mathcal{A}}M_{T(\mathbf{1})}$ is closed, then by Lemma \ref{l1} we get that $E^{\mathcal{A}}M_{T(\mathbf{1})}$ is bounded and $T=E^{\mathcal{A}}M_{T(\mathbf{1})}$. Therefore $\mathcal{W}$ is strongly closed and consequently is a unital commutative Von Neumann algebra with unit $E^{\mathcal{A}}$.
\end{proof}
Let $M$ be a closed subspace of $L^2(\mathcal{F})$ that reduces $E^{\mathcal{A}}M_u$. Let $\mathcal{B}_0$ be the $\sigma$- subalgebra generated by $\{f^{-1}(\mathcal{C}): f\in M\}\cup u^{-1}(\mathcal{C})$, $\mathcal{B}_1$ be the $\sigma$-algebra generated by $\{f^{-1}(\mathcal{C}): f\in M\}\cup u^{-1}(\mathcal{C})\cup \mathcal{A}$ and $\mathcal{B}_2$ be the $\sigma$-algebra generated by $\{f^{-1}(\mathcal{C}): f\in M\}\cup \mathcal{A}$. Then $\mathcal{B}_0\subseteq \mathcal{B}_1$, $M\subseteq L^2(\mathcal{B}_0)\subseteq L^2(\mathcal{B}_1)$ and the subspaces $L^2(\mathcal{B}_0)$,  $L^2(\mathcal{B}_1)$ are reducing subspaces for $E^{\mathcal{A}}M_u$. Also, we have $\mathcal{B}_0\subseteq \mathcal{B}_2\subseteq \mathcal{B}_1$, if $u$ is $\mathcal{A}$-measurable, moreover $L^2(\mathcal{B}_2)$  reduces $E^{\mathcal{A}}M_u$.\\

In the sequel we provide some necessary and sufficient conditions for a reducing subspace to be of the form $L^2(A)$ or $L^2(\mathcal{B})$.
\begin{thm}\label{t6} Let $(X, \mathcal{F}, \mu)$ be a finite measure space, $M\subseteq L^2(\mathcal{F})$ be a reducing subspace for $E^{\mathcal{A}}M_u$, $P$ be the orthogonal projection onto $M$ and $p=P(1)$. Then there exists a $\sigma$- algebra $\mathcal{M}\subseteq \mathcal{A}$ such that $L^2(\mathcal{M})$ reduces $E^{\mathcal{A}}M_u$ and the followings hold:\\

(a) If $S(E^{\mathcal{A}}(|u|^2))=X$, then $Pf=E^{\mathcal{A}}(p)f$, for all $f\in L^2(\mathcal{M})$.\\

(b) If $M= L^2(A)$ for some $A\in \mathcal{F}$, then $\mathcal{M}$ is equal to the $\sigma$- algebra generated by $\mathcal{A}_{A}$, in which $\mathcal{A}_{A}=\{B\in \mathcal{A}: A\cap B\in \mathcal{A}\}$.\\

(c) If $\mathcal{M}=\mathcal{A}$, then $E^{\mathcal{A}}(M)=f.L^2(\mathcal{A})\subseteq L^1(S(f))$, for some $f\in L^2(\mathcal{A})$.\\

(d) $L^2(\mathcal{M})\subseteq M$ if and only if $1\in M$.\\

(e) $L^2(\mathcal{M})\subseteq M^{\perp}$ if and only if $1\in M^{\perp}$.

\end{thm}
\begin{proof} Put
$$W=\{E^{\mathcal{A}}M_g: g\in L^{\infty}(\mathcal{A}), \ \ \ \ E^{\mathcal{A}}M_gP=PE^{\mathcal{A}}M_g\}.$$

It is easy to see that $W$ is a Von Neumann subalgebra of $\{E^{\mathcal{A}}M_g: g\in L^{\infty}(\mathcal{A})\}$. Let
$$\mathcal{M}_0=\{A\in \mathcal{A}: E^{\mathcal{A}}M_{\chi_A}\in W\}.$$
Easily we have that $\emptyset \in \mathcal{M}_0$ and $A\cap B\in \mathcal{M}_0$, for all $A, B\in \mathcal{M}_0$.  Also we have $W=\{E^{\mathcal{A}}M_g: g\in L^{\infty}(\mathcal{M}_0)\}$. If we put $\mathcal{M}=\mathcal{M}_0\cup \{X\}$, then $\mathcal{M}$ is a $\sigma$- subalgebra of $\mathcal{A}$. Also, $\mathcal{M}_0$ is a $\sigma$- algebra when $E^{\mathcal{A}}P=PE^{\mathcal{A}}$, in this case we set $\mathcal{M}=\mathcal{M}_0$.  In general we suppose that $\mathcal{M}$ is the $\sigma$- algebra generated by $\mathcal{M}_0$. Now we show that $L^2(\mathcal{M})$ reduces $E^{\mathcal{A}}M_u$, for this, it suffices to show that  $L^{\infty}(\mathcal{M}_0)$ reduces $E^{\mathcal{A}}M_u$. Let $g\in L^{\infty}(\mathcal{M}_0)$ and $f\in L^{2}(\mathcal{F})$. Then
\begin{align*}
E^{\mathcal{A}}M_{E^{\mathcal{A}}(u)g}Pf&=E^{\mathcal{A}}M_uE^{\mathcal{A}}M_gPf\\
&=E^{\mathcal{A}}M_uPE^{\mathcal{A}}M_gf\\
&=PE^{\mathcal{A}}M_uPE^{\mathcal{A}}M_gf\\
&=PE^{\mathcal{A}}M_{E^{\mathcal{A}}(u)g}f.
\end{align*}
Also we have
\begin{align*}
PE^{\mathcal{A}}M_{ug}f&=PE^{\mathcal{A}}M_gE^{\mathcal{A}}M_uf\\
&=E^{\mathcal{A}}M_gPE^{\mathcal{A}}M_uf\\
&=E^{\mathcal{A}}M_gE^{\mathcal{A}}M_uPf\\
&=E^{\mathcal{A}}M_{ug}Pf.
\end{align*}
By these observations  we get that $L^{\infty}(\mathcal{M}_0)$ is invariant under $E^{\mathcal{A}}M_u$ and $M_uE^{\mathcal{A}}$. Hence we get that $L^{\infty}(\mathcal{M})$ reduces $E^{\mathcal{A}}M_u$. Consequently, $L^{2}(\mathcal{M})$ reduces $E^{\mathcal{A}}M_u$\\

 Also for $f\in L^{2}(\mathcal{M}_0)$ we have
\begin{align*}
P(f)&=PE^{\mathcal{A}}(f.1)\\
&=PE^{\mathcal{A}}M_f(1)\\
&=E^{\mathcal{A}}M_fP(1)\\
&=f.E^{\mathcal{A}}(P(1))\\
&=M_{E^{\mathcal{A}}(P(1))}(f).
\end{align*}
Hence we have (a).\\

(b) Suppose that $M=L^2(A)$ for some $A\in \mathcal{F}$. Then $P=M_{\chi_A}$ and $\mathcal{M}_0=\{B\in \mathcal{A}: A\cap B\in \mathcal{A}\}=\mathcal{A}_A$. Thus $\mathcal{M}$ is equal to the $\sigma$- algebra generated by $\mathcal{A}_A$.\\

(c) If $\mathcal{M}=\mathcal{A}$, then by part (a) for every $f\in L^2(\mathcal{A})$ we have $P(f)=E^{\mathcal{A}}(p)f$. Also, for every $B\in \mathcal{A}\setminus \{X\}$ we have $B\in \mathcal{M}_0$ and so for all $f\in L^2(\mathcal{F})$, $$E^{\mathcal{A}}(\chi_BP(f))=P(E^{\mathcal{A}}(\chi_Bf))=E^{\mathcal{A}}(p)\chi_BE^{\mathcal{A}}(f).$$
This implies that  $E^{\mathcal{A}}P(f)=E^{\mathcal{A}}(p)E^{\mathcal{A}}(f)$ for all $f\in L^2(\mathcal{F})$. Therefore $E^{\mathcal{A}}(M)=E^{\mathcal{A}}(p).L^2(\mathcal{A})$.\\

(d) If $L^2(\mathcal{M})\subseteq M$, then $1\in L^2(\mathcal{M})\subseteq M$. Conversely, if $1\in M$, then $p=P(1)=1$ and so $E^{\mathcal{A}}(p)=E^{\mathcal{A}}(1)=1$. Hence by (a) we get that $P(f)=f$ for all $f\in L^2(M)$. This means that $L^2(\mathcal{M})\subseteq M$.\\

(e) If $L^2(\mathcal{M})\subseteq M^{\perp}$, then $1\in L^2(\mathcal{M})\subseteq M^{\perp}$. Conversely, if $1\in M^{\perp}$, then $p=P(1)=0$ and so $P(f)=0.f$, for all $f\in L^2(M)$. This implies that $L^2(\mathcal{M})\subseteq M^{\perp}$.

\end{proof}

\end{document}